\documentclass{amsart}

\usepackage{mathptmx,hyperref}

\newtheorem*{qst}{Question}
\newtheorem*{thm}{Theorem}
\newtheorem*{cor}{Corollary}
\newtheorem*{prp}{Proposition}

\begin{document}

\title{Free actions on surfaces that do not extend to arbitrary actions on 3-manifolds}
\author{Eric G.\ Samperton}
\thanks{The author thanks Carlos Segovia, Jes\'us Emilio Dom\'inguez, Nathan Dunfield and Bernardo Uribe for helpful conversations, and acknowledges support from the National Science Foundation via grant DMS \#2038020.}
\date{January 4, 2022}

\begin{abstract} 
We provide the first known example of a finite group action on an oriented surface $T$ that is free, orientation-preserving, and does not extend to an arbitrary (in particular, possibly non-free) orientation-preserving action on any compact oriented 3-manifold $N$ with boundary $\partial N = T$. This implies a negative solution to a conjecture of Dom\'inguez and Segovia, as well as Uribe's evenness conjecture for equivariant unitary bordism groups.  We more generally provide sufficient conditions implying that infinitely many such group actions on surfaces exist.  Intriguingly, any group with such a non-extending action is also a counterexample to the Noether problem over the complex numbers $\mathbb{C}$.  In forthcoming work with Segovia we give a complete homological characterization of those finite groups admitting such a non-extending action, as well as more examples and non-examples.  We do not address here the analogous question for non-orientation-preserving actions.
\end{abstract}

\maketitle

\section{Introduction}

\subsection*{Main results}
Let $G$ be a finite group.  Dom\'inguez and Segovia recently asked the following \cite{DominguezSegovia:Schur}:

\begin{qst}[Free extension conjecture]
Suppose $G$ acts freely on an oriented surface $T$ by orientation-preserving homeomorphisms.  Does there always exist a compact oriented 3-manifold $N$ with boundary $\partial N = T$ and an orientation-preserving action of $G$ on $N$---possibly non-free---that extends the given action on $T$?
\end{qst}

\noindent They showed the answer to this question is yes for several types of groups $G$ with very different properties, including abelian groups, dihedral groups, alternating groups and symmetric groups.  Accordingly, they conjectured that it is always true.

In this short note we provide the first counterexamples.  The smallest confirmed counterexamples we could find have order $3^5 = 243$.  We initially found them using a systematic search with GAP.  Using the notation of GAP's small group library \cite{GAP}, the counterexamples are SmallGroup(243, 28), SmallGroup(243, 29) and SmallGroup(243, 30).  They are isoclinic.  We later realized that work of Moravec \cite{Moravec:unramified} and others \cite{HoshiKangKunyavskii:unramified} on the calculation of Bogomolov multipliers allows for a computer-free proof that infinitely many counterexamples exist.  It is still possible that there are smaller counterexamples, but they would have to be of order $2^6=64$, $2^7=128$ or $2^6 \times 3 = 192$.  We leave a more careful study of such groups to more expansive forthcoming work with Segovia.

\subsection*{Related results and background} Before diving in, we briefly remark on the history of this question.  It surprised us to learn that \cite{DominguezSegovia:Schur} seems to be the first time this precise question appears in the literature.  In personal communications, Segovia has told us that Bernardo Uribe first asked it of him, and Uribe in turn informed us that he was motivated by the evenness conjecture for equivariant unitary bordism groups, which he recently discussed in the 2018 ICM Proceedings \cite{Uribe:evenness}.  This conjecture was originally posed as a problem by Rowlett in 1978 \cite{Rowlett:metacyclic}: is the $G$-equivariant unitary bordism ring $\Omega_*^U(G)$ always a free module over the usual unitary bordism ring $\Omega_*^U$ in even generators?  This is known to be true for all finite abelian groups \cite{Stong:actions,Ossa:abelian} and all finite metacyclic groups \cite{Rowlett:metacyclic}, and Uribe conjectured it is true in general \cite{Uribe:evenness}.  However, if $G$ satisfies the evenness conjecture, then $G$ must satisfy the free extension conjecture \cite{Uribe:personal} (see also forthcoming joint work with Angel, Segovia and Uribe).  Thus, our counterexamples to the latter are also counterexamples to the former.

We note that some of the results of \cite{DominguezSegovia:Schur} were shown previously by Hidalgo \cite{Hidalgo:Schottky}, and Reni and Zimmerman \cite{ReniZimmerman:handlebodies}.  Specifically, \cite{Hidalgo:Schottky,ReniZimmerman:handlebodies} both show that every orientation-preserving action (not necessarily free) of an abelian group on a surface extends to an action on a handlebody, and \cite{ReniZimmerman:handlebodies} shows the same for dihedral group actions.

A number of other related results can be found in \cite{ReniZimmerman:handlebodies}.  For instance, building on \cite{GradolatoZimmerman:hyperbolic}, they show that the only $84(g-1)$ Hurwitz actions of type $\mathrm{PSL}(2,q)$ with $7 \le q < 1000$ that do not extend to any 3-manifold occur when $q=7$ and $q=27$.  Hurwitz actions are not free.  As they note and we shall see, ``it is somewhat easier to construct non-extending group actions on surfaces with fixed points because one can use the structure of the singular set in 3-orbifolds."

The closest result to our counterexamples we could find is \cite[Prop.~1]{ReniZimmerman:handlebodies} and the examples that follow.  Their proposition gives sufficient conditions to guarantee that a group $G$ admits a free action on a surface that does not extend to any action \emph{on a handlebody}.  They note, ``we do not know if these actions extend to some other compact 3-manifold; in fact, at present we do not know any example of a free action which extends to a compact 3-manifold but not to a handlebody."  To keep this note short we address neither of the two questions implicit here.  Namely: Do the examples of free actions in \cite[Prop.~1]{ReniZimmerman:handlebodies} extend to 3-manifolds that aren't handlebodies?  Does every free action that extends extend to a handlebody?

We also do not attempt to address another very natural question: do there exist free actions on surfaces that do not extend even over a \emph{non-oriented} 3-manifold?

\section{Obstructions to extendability}
We begin in earnest now with some definitions.  Let us set the conventions that for the rest of this note all 2- and 3-manifolds are oriented and all actions on them are orientation-preserving.

We say a free action of a group $G$ on a surface $T$ \emph{extends} if there exists a 3-manifold $N$ with $\partial N = T$ and an action of $G$ on $N$ that extends the given action on $T$.  If there exists such an $N$ where the $G$ action is free, then we will say the action of $G$ on $T$ \emph{freely extends}.

In between ``freely extends" and ``extends" we have the intermediate notion of ``non-singularly extends."  If $G$ acts on $N$ non-freely, then the \emph{ramification locus} is the set of points in $N$ stabilized by non-trivial elements of $G$.  The \emph{branch locus} is the image of the ramification locus in the quotient space $N/G$.  An action of $G$ on $N$ is called \emph{non-singular} if the branch locus is either empty or a codimension 2 submanifold (that is, a link inside $N/G$).  We say a (free) action of $G$ on $T$ \emph{non-singularly extends} if there exists a manifold $N$ with $\partial N = T$ and a non-singular action of $G$ on $N$ that extends the given action on $T$.  An action that extends neither freely nor non-singularly might be said to extend \emph{only singularly}.

By \emph{spherical subgroup} of $G$ we mean any subgroup isomorphic to the fundamental group of an orientable spherical 2-orbifold with 3 cone points.  That is, a spherical subgroup of $G$ is any subgroup isomorphic to a non-cyclic subgroup of $SO(3)$.  Even more concretely, the spherical subgroups of $G$ are those that are isomorphic to one of the following: a dihedral group $D_{2n}$, the alternating groups $A_4$ or $A_5$, or the symmetric group $S_4$.

We let $BG$ denote the classifying space of $G$.  The \emph{Schur multiplier} of $G$ is $M(G) := H_2(BG)$.  A \emph{toral class} in $M(G)$ is any element that can be represented by a map from a torus $S^1 \times S^1$ into $BG$.

\begin{thm}
Let $G$ be a finite group.
\begin{enumerate}
\item Every free surface action of $G$ freely extends if and only if $M(G) = 0$.
\item Every free surface action of $G$ extends non-singularly if and only if $M(G)$ is generated by toral classes.
\item If $M(G)$ is not generated by toral classes and $G$ has no spherical subgroups, then $G$ affords free surface actions that do not extend.
\end{enumerate}
\label{th:main}
\end{thm}

\begin{proof}
\begin{enumerate}
\item This well-known fact follows because oriented bordism and homology are isomorphic in dimension 2.  We sketch the argument as a warm-up to the next part of the proof.

Let $G$ act freely on $T$ and let $S := T/G$ be the quotient surface.  Since $T$ is a regular $G$-cover of $S$, there exists a classifying map $f: S \to BG$ where $BG$ is the classifying space of $G$.  Let $[S] \in H_2(S)$ be the orientation class of $S$ induced from the orientation of $T$.  Then $f_*[S] = 0 \in M(G) = H_2(BG) \cong \Omega_2(BG)$ if and only if the map $f$ extends to a map $\hat{f}: M \to BG$ where $M$ is an oriented 3-manifold $M$ with $\partial M = S$.  Let $N$ be the $G$-cover of $M$ determined by $\hat{f}$.  Then the $G$ action on $N$ is free and extends the $G$ action on $\partial N = T$.

\item Just as we proved part (1) by looking downstairs in the quotient, we do the same with this part.  Our key tool is a certain classifying space for nonsingular branched covers studied in \cite{me:Schur}.

Following Brand \cite{Brand:branched} and Ellenberg, Venkatesh and Westerland \cite{EVW:Hurwitz2}, define
\[ BG_G := BG \bigsqcup_{ev} LBG \times D^2\]
where $LBG=\mathrm{Map}(S^1,BG)$ is the free loop space of $BG$ and $ev$ is the evaluation map
\[ ev: LBG \times \partial D^2 = LBG \times S^1 \to BG. \]
Intuitively, we construct $BG_G$ by gluing a disk to $BG$ along every map from a circle into $BG$.

Theorem 2.2 of \cite{me:Schur} shows that for any (smooth) manifold $M$ (homotopy classes of) maps $M \to BG_G$ are in natural bijection with (concordance classes of) branched $G$-covers of $M$ with branch locus a framed codimension 2 submanifold $L \subset M$.  In particular, $H_2(BG_G)$ is isomorphic to a flavor of 2-dimensional non-singular branched $G$-bordism group, although we do not need to make this precise.  

We now proceed as in part (1).  Suppose $G$ acts freely on $T$ and let $S := T/G$ be the quotient surface.  Since $T$ is a regular $G$-cover of $S$, it is also a branched $G$-cover of $S$ (with empty branch locus), and so \cite[Thm.~2.2]{me:Schur} says there exists a classifying map $f': S \to BG_G$.  Because homology and oriented bordism are the same in dimension 2, $f'_*[S] = 0 \in H_2(BG_G) \cong \Omega_2(BG_G)$ if and only if the map $f'$ extends to a map $\widehat{f'}: M \to BG_G$ where $M$ is an oriented 3-manifold with $\partial M = S$.  Let $N$ be the branched $G$-cover of $M$ determined from $\widehat{f'}$ by \cite[Thm.~2.2]{me:Schur}.  Then the $G$ action on $N$ is nonsingular and extends the $G$ action on $\partial N = T$.

Of course the inclusion $i: BG \to BG_G$ induces a map
\[i_*: M(G) = H_2(BG) \to H_2(BG_G) \]
and we have both $f' = i \circ f$ and $f'_* = i_* \circ f_*$.  Thus every surface action of $G$ extends non-singularly if and only if $i_* = 0$.  Finally, an easy calculation using the Mayer-Vietoris sequence (carried out in detail in \cite[Sec.~2.3]{me:Schur}) shows $\mathrm{ker} \ i_*$ is precisely the subgroup of $M(G)$ generated by toral classes.

\item Having proved the previous part, this is now a standard argument used in the study of orbifolds, see e.g. \cite[Thms.~2.3 \& 2.5]{CooperHodgsonKerckhoff:orbifolds}.  Suppose $M(G)$ is not generated by toral classes.  Then by part (2) there exists a free action of $G$ on some surface $T$ that does not extend non-singularly.  Suppose this action extends singularly to some 3-manifold $N$.  We will show that $G$ must have a spherical subgroup.

Smoothing both $N$ and the action of $G$, we can find a $G$-invariant Riemannian metric on $N$ (e.g. after averaging over $G$).  By assumption, the ramification locus in $N$ is nonempty and is not a codimension 2 submanifold.  Suppose $x \in N$ has non-trivial stabilizer $G_x \le G$.  If $U \subset N$ is a small enough $G_x$-invariant ball centered at $x$, then $U$ is homeomorphic (via the inverse of the exponential map) to a $G_x$ invariant ball centered at $0$ in the tangent space $T_xN$.  The metric is $G$-invariant and $G$ preserves the orientation of $N$, so $G_x$ acts on $T_xN$ and hence $\partial U \cong S^2$ as a subgroup of $SO(3)$.  Moreover, the portion of the ramification locus inside $U$ is homeomorphic to the cone over the non-trivially stabilized points in $\partial U \cong S^2$.  Since by assumption the ramification locus is not a manifold, we can pick an $x$ such that $G_x$ acts on $\partial U$ with more than 2 non-trivially stabilized points and hence is non-cyclic.  Thus, we can pick an $x$ such that $G_x$ is a spherical subgroup.
\end{enumerate}
\vspace{-\baselineskip}
\end{proof}

Spherical orbifold groups all have even order.  So we have an obvious corollary.

\begin{cor}
If $G$ is a finite group of odd order and $M(G)$ is not generated by toral classes, then $G$ affords free surface actions that do not extend. \qed
\label{c:sufficient}
\end{cor}

Therefore, to find a free action of a finite group on a surface that does not extend over a compact 3-manifold, it suffices to find an odd order group $G$ such that $M(G)$ is not generated by tori.

\section{Counterexamples to the free extension conjecture}

\subsection*{Finding counterexamples with GAP}
We describe here functions written for GAP that we used to perform a systematic search for (odd order) counterexamples to the free extension conjecture.  The code can be downloaded from the author's University of Illinois webpage at
\begin{center}\url{https://smprtn.pages.math.illinois.edu/ToralProbe.gap}\end{center}
After writing it, we learned that the branched Schur multiplier is isomorphic to the Bogomolov multiplier (see below).  Hence, our code has similar functionality to GAP code written by Moravec \cite{Moravec:unramified}, with two caveats.  On one hand, Jezernik and Moravec later wrote better optimized code using a novel algorithm \cite{JezernikMoravec:128}, and it appears to outperform ours slightly.  On the other hand, our code also handles non-polycyclic (hence, non-solvable) groups.

The most useful function is ToralProbe(G), which takes an input group G and outputs a GAP record of the form (group = G, multiplier =$|M(G)|$, order of subgroup generated by toral classes, test Boolean multiplier =? order).  For example, running ToralProbe(SmallGroup(243,28)) outputs the record rec( group := $\langle$pc group of size 243 with 5 generators$\rangle$, multiplier := 9, order := 3, test := false ).  This calculation took 119 milliseconds on a 2016 MacBook with a 1.1 GHz Intel Core m3 processor.

We also include the function NaivePcGenusTwoProbe(G).  The input must be a polycyclic group $G$ such as SmallGroup(243,28).  It returns a record similar to ToralProbe(G), except that it determines the order of the subgroup of $M(G)$ generated by both toral classes and genus 2 surfaces.  For example, NaivePcGenusTwoProbe(SmallGroup(243,28)) returns rec(group := $\langle$pc group of size 243 with 5 generators$\rangle$, multiplier := 9, order := 9, genustwotest := true ).  This took our computer 15 seconds.

An especially helpful function is SearchThroughRangeOfSmallGroupOrders(a,b).  For all $G$ with $a \le |G| \le b$ whose Schur multiplier is not generated by tori, this prints $|G|$, the index of $G$ in the small group library, $|M(G)|$, and the order of the toral subgroup of $M(G)$.  Thus the easiest way to verify our results is to run Read(``ToralProbe.gap") and then run SearchThroughRangeOfSmallGroupOrders(1,243).  The resulting output indicates all 296 examples of groups of order less than 244 whose Schur multiplier is not generated by tori.  There are 9 of order 64, 230 of order 128, 54 of order 192 and 3 of order 243.  This calculation took 12.5 minutes on our computer.

Further calculations show that $M(G)$ is generated by surfaces of genus at most 2 for all 296 of these groups.  In particular, this means that for each of the 3 groups of order 243, the smallest non-extending free actions are on surfaces with Euler characteristic $-2 \times 243 = -486$, that is, on surfaces of genus $244$.  Of course, this means that the quotient surfaces of these actions have genus 2.  A description of the monodromy could be calculated easily with GAP, but we do not include it here.

We note it is entirely possible that some of the groups of even orders 64, 128 and 192 found by this calculation afford non-extending free surface actions.  By part 2 of our Theorem, it is certainly true that these groups afford free actions on surfaces that do not extend either freely or non-singularly.  However, further work must be done to determine if they have spherical subgroups that can be used to extend the actions singularly.

\subsection*{Examples from birational complex geometry}
In the previous section we provided a brief description of code written in GAP that allows one to find counterexample groups systematically.  Here we offer a computer-free alternative that exhibits intriguing connections to complex algebraic geometry.

In our past work \cite{me:Schur}, we defined the \emph{non-singular $G$-branched Schur multiplier} $M(G)_G$ to be the quotient of the Schur multiplier $M(G)$ by the subgroup generated by toral classes.  In this language, our above theorem shows $M(G)_G = 0$ if and only if every free surface action of $G$ extends non-singularly.  We recently learned that in a rather different context, Moravec showed that $M(G)_G$ is isomorphic (non-canonically) to the \emph{Bogomolov multiplier} $B_0(G)$ \cite{Moravec:unramified}, a well-studied cohomological invariant of $G$ pertinent to the complex Noether problem.  We review these ideas briefly now.

The Noether problem over the field $k$ asks: given a faithful $k$-representation $V$ of the finite group $G$, do the $G$-invariant rational functions $k(V)^G$ form a purely transcendental extension of $k$?  Or, in equivalent algebro-geometric terms: is $V/G$ a rational algebraic variety over $k$?  Recall that $k(V)^G$ is \emph{purely transcendental} if it is isomorphic to a field of rational functions in some finite number of variables, and $V/G$ is a \emph{rational} variety if some Zariski open subset is isomorphic to a Zariski open subset of a projective space $\mathbb{P}_k^d$.

The Bogomolov multiplier has played a central role in the discovery of counterexamples to the Noether problem over $\mathbb{C}$. Indeed, Saltman provided the first counterexamples by showing that if the \emph{unramified cohomology group} $H^2_{nr}(\mathbb{C}(V)^G,\mathbb{Q}/\mathbb{Z})$ is nonzero, then $G$ has a negative solution to the Noether problem over $\mathbb{C}$ \cite{Saltman:Noether}.  Shortly thereafter, Bogomolov showed $H^2_{nr}(\mathbb{C}(V)^G,\mathbb{Q}/\mathbb{Z})$ is naturally isomorphic to the subgroup $B_0(G)$ of $H^2(G,\mathbb{Q}/\mathbb{Z})$ consisting of classes that vanish when restricted to the abelian subgroups of $G$ \cite{Bogomolov:Brauer}.  The simpler, group-cohomological definition of $B_0(G)$ allowed Bogomolov to find many more counterexamples to the complex Noether problem.   Much later, Kunyavski\u{i} gave $B_0(G)$ the name \emph{Bogomolov multiplier} \cite{Kunyavskii:Bogomolov}.

More recently, Moravec gave a group-\emph{homological} description of $B_0(G)$.  Precisely, he showed that $B_0(G)$ is (non-canonically) isomorphic to the quotient of the Schur multiplier $M(G)$ by the subgroup represented by commutators of commuting elements of $G$---that is, by toral classes \cite[Lem.~3.1 \& Thm.~3.2]{Moravec:unramified}.  Thus, the Bogomolov multiplier $B_0(G)$ is (non-canonically) isomorphic to the non-singular $G$-branched Schur multiplier $M(G)_G$ we defined in \cite{me:Schur}.

Combining our work with the algebraic geometers' we arrive in a peculiar place.

\begin{prp}
If $G$ affords free surface actions that do not extend non-singularly, then $G$ is a counterexample to the Noether problem over $ \mathbb{ C}$. \qed
\end{prp}

We would very much like to know if there is an elementary proof of this proposition.

In any case, we are now in the convenient position where we can find counterexamples to the free extension conjecture by mining the literature on Bogomolov multipliers.  We just need odd order groups with non-vanishing $B_0(G)$.  We find such gems in several papers.  For example, Moravec uses GAP to show that 3 of the groups of order $3^5 = 243$ and 85 of the groups of order $3^6 = 729$ have non-vanishing $B_0(G)$ \cite{Moravec:unramified}.  Calculations with our own similar GAP functions (which were written independently before we learned of Moravec's work) described in the appendix confirm these results.  For the computer-skeptical, results of \cite{HoshiKangKunyavskii:unramified} show ``by hand" that for every odd prime $p$, there exist groups of order $p^5$ with non-trivial Bogomolov multipliers, and, hence, admit non-extending actions.  More counterexamples can be mined, but we leave them to future work.

\section{Concluding remarks}
In this short note our goal has been to describe the most direct route to establishing counterexamples to the free extension conjecture, while still providing some inspiration for many new paths to explore.  In forthcoming work with Segovia, we initiate a more thorough study of non-extending group actions on surfaces.  Some of our main results include:
\begin{itemize}
\item A complete homological characterization of finite group actions on surfaces (not necessarily free) that do not extend.
\item A proof that every free action of a finite simple group extends non-singularly. Likewise for finite Coxeter groups.
\end{itemize}

We conclude with a brief list of problems, some of which we hope to solve soon.
\begin{itemize}
\item Understand extendability questions for non-orientation-preserving actions.
\item Classify non-extending actions of $p$-groups.
\item Does there exist a non-solvable finite group that affords non-extending free actions?
\item Does there exist an action that extends---but only singularly?
\item Develop elementary connections between extendibility of actions and the Noether problem over $\mathbb{C}$.  In particular, is it true that if $G$ is a counterexample to the complex Noether problem, then $G$ affords non-extending actions?
\end{itemize}


\begin{thebibliography}{EVW13}
\expandafter\ifx\csname url\endcsname\relax
  \def\url#1{\texttt{#1}}\fi
\expandafter\ifx\csname doi\endcsname\relax
  \def\doi#1{\burlalt{doi:#1}{http://dx.doi.org/#1}}\fi
\expandafter\ifx\csname urlprefix\endcsname\relax\def\urlprefix{URL }\fi
\expandafter\ifx\csname href\endcsname\relax
  \def\href#1#2{#2}\fi
\expandafter\ifx\csname burlalt\endcsname\relax
  \def\burlalt#1#2{\href{#2}{#1}}\fi

\bibitem[Bog87]{Bogomolov:Brauer}
F.~A. Bogomolov.
\newblock The {B}rauer group of quotient spaces of linear representations.
\newblock {\em Izv. Akad. Nauk SSSR Ser. Mat.}, 51(3):485--516, 688, 1987.
\newblock \doi{10.1070/IM1988v030n03ABEH001024}.

\bibitem[Bra80]{Brand:branched}
Neal Brand.
\newblock Classifying spaces for branched coverings.
\newblock {\em Indiana Univ. Math. J.}, 29(2):229--248, 1980.
\newblock \doi{10.1512/iumj.1980.29.29015}.

\bibitem[CHK00]{CooperHodgsonKerckhoff:orbifolds}
Daryl Cooper, Craig~D. Hodgson, and Steven~P. Kerckhoff.
\newblock {\em Three-dimensional orbifolds and cone-manifolds}, volume~5 of
  {\em MSJ Memoirs}.
\newblock Mathematical Society of Japan, Tokyo, 2000.
\newblock With a postface by Sadayoshi Kojima.

\bibitem[DS21]{DominguezSegovia:Schur}
Jes\'{u}s~Emilio Dom\'{i}nguez and Carlos Segovia.
\newblock Extending free group action on surfaces, 2021,
  \burlalt{arXiv:2012.02464}{http://arxiv.org/abs/arXiv:2012.02464}.

\bibitem[EVW13]{EVW:Hurwitz2}
Jordan~S. Ellenberg, Akshay Venkatesh, and Craig Westerland.
\newblock Homological stability for {Hurwitz} spaces and the {Cohen-Lenstra}
  conjecture over function fields, {II}, 2013,
  \burlalt{arXiv:1212.0923v1}{http://arxiv.org/abs/arXiv:1212.0923v1}.

\bibitem[GAP21]{GAP}
The GAP~Group.
\newblock {\em {GAP -- Groups, Algorithms, and Programming, Version 4.11.1}},
  2021.
\newblock \urlprefix\url{https://www.gap-system.org}.

\bibitem[GZ95]{GradolatoZimmerman:hyperbolic}
Monique Gradolato and Bruno Zimmermann.
\newblock Extending finite group actions on surfaces to hyperbolic
  {$3$}-manifolds.
\newblock {\em Math. Proc. Cambridge Philos. Soc.}, 117(1):137--151, 1995.
\newblock \doi{10.1017/S0305004100072960}.

\bibitem[Hid94]{Hidalgo:Schottky}
Rub\'{e}n~A. Hidalgo.
\newblock On {S}chottky groups with automorphisms.
\newblock {\em Ann. Acad. Sci. Fenn. Ser. A I Math.}, 19(2):259--289, 1994.
\newblock
  \urlprefix\url{http://www.acadsci.fi/mathematica/Vol19/hidalgo2.html}.

\bibitem[HKK13]{HoshiKangKunyavskii:unramified}
Akinari Hoshi, Ming-Chang Kang, and Boris~E. Kunyavskii.
\newblock Noether's problem and unramified {B}rauer groups.
\newblock {\em Asian J. Math.}, 17(4):689--713, 2013.
\newblock \doi{10.4310/AJM.2013.v17.n4.a8}.

\bibitem[JM14]{JezernikMoravec:128}
Urban Jezernik and Primo\v{z} Moravec.
\newblock Bogomolov multipliers of groups of order 128.
\newblock {\em Exp. Math.}, 23(2):174--180, 2014.
\newblock \doi{10.1080/10586458.2014.886980}.

\bibitem[Kun10]{Kunyavskii:Bogomolov}
Boris Kunyavski\u{\i}.
\newblock The {B}ogomolov multiplier of finite simple groups.
\newblock In {\em Cohomological and geometric approaches to rationality
  problems}, volume 282 of {\em Progr. Math.}, pages 209--217. Birkh\"{a}user
  Boston, Boston, MA, 2010.
\newblock \doi{10.1007/978-0-8176-4934-0\_8}.

\bibitem[Mor12]{Moravec:unramified}
Primo\v{z} Moravec.
\newblock Unramified {B}rauer groups of finite and infinite groups.
\newblock {\em Amer. J. Math.}, 134(6):1679--1704, 2012.
\newblock \doi{10.1353/ajm.2012.0046}.

\bibitem[Oss72]{Ossa:abelian}
Erich Ossa.
\newblock Unitary bordism of abelian groups.
\newblock {\em Proc. Amer. Math. Soc.}, 33:568--571, 1972.
\newblock \doi{10.2307/2038101}.

\bibitem[Row80]{Rowlett:metacyclic}
Russell~J. Rowlett.
\newblock Bordism of metacyclic group actions.
\newblock {\em Michigan Math. J.}, 27(2):223--233, 1980.
\newblock \urlprefix\url{http://projecteuclid.org/euclid.mmj/1029002359}.

\bibitem[RZ96]{ReniZimmerman:handlebodies}
Marco Reni and Bruno Zimmermann.
\newblock Extending finite group actions from surfaces to handlebodies.
\newblock {\em Proc. Amer. Math. Soc.}, 124(9):2877--2887, 1996.
\newblock \doi{10.1090/S0002-9939-96-03515-0}.

\bibitem[Sal84]{Saltman:Noether}
David~J. Saltman.
\newblock Noether's problem over an algebraically closed field.
\newblock {\em Invent. Math.}, 77(1):71--84, 1984.
\newblock \doi{10.1007/BF01389135}.

\bibitem[Sam20]{me:Schur}
Eric Samperton.
\newblock Schur-type invariants of branched {$G$}-covers of surfaces.
\newblock In {\em Topological phases of matter and quantum computation}, volume
  747 of {\em Contemp. Math.}, pages 173--197. Amer. Math. Soc., Providence,
  RI, 2020, \burlalt{arXiv:1709.03182}{http://arxiv.org/abs/arXiv:1709.03182}.
\newblock \doi{10.1090/conm/747/15045}.

\bibitem[Sto70]{Stong:actions}
R.~E. Stong.
\newblock {\em Unoriented bordism and actions of finite groups}.
\newblock Memoirs of the American Mathematical Society, No. 103. American
  Mathematical Society, Providence, R.I., 1970.

\bibitem[Uri18]{Uribe:evenness}
Bernardo Uribe.
\newblock The evenness conjecture in equivariant unitary bordism.
\newblock In {\em Proceedings of the {I}nternational {C}ongress of
  {M}athematicians---{R}io de {J}aneiro 2018. {V}ol. {II}. {I}nvited lectures},
  pages 1217--1239. World Sci. Publ., Hackensack, NJ, 2018.

\bibitem[Uri21]{Uribe:personal}
Bernardo Uribe, 2021.
\newblock Private communication.

\end{thebibliography}
\end{document}